\documentclass[10pt]{amsart}
\usepackage{amssymb, amsmath}
\usepackage{graphics,xcolor}
\usepackage{latexsym,amsmath,amssymb,amsthm,amsfonts,mathrsfs}
\usepackage{amscd}
\usepackage{tikz-cd}
\usepackage[arrow, matrix, curve]{xy}

\DeclareMathOperator{\en}{End}

\DeclareMathOperator{\Sp}{Sp}

\DeclareMathOperator{\ima}{im}

\DeclareMathOperator{\mon}{Mon}
\DeclareMathOperator{\mi}{\textbf{m}}
\theoremstyle{plain}
\newtheorem{thm}{Theorem}[section]
\newtheorem{theorem}[thm]{Theorem}
\newtheorem{lemma}[thm]{Lemma}

\newtheorem{proposition}[thm]{Proposition}

\theoremstyle{definition}

\newtheorem{remark}[thm]{Remark}

\newtheorem{definition}[thm]{Definition}

\numberwithin{equation}{thm}

\newcommand{\sC}{{\mathcal C}}

\newcommand{\sL}{{\mathcal L}}
\newcommand{\sM}{{\mathcal M}}

\newcommand{\A}{{\mathbb A}}

\newcommand{\C}{{\mathbb C}}

\renewcommand{\H}{{\mathbb H}}

\renewcommand{\L}{{\mathbb L}}

\renewcommand{\P}{{\mathbb P}}
\newcommand{\Q}{{\mathbb Q}}
\newcommand{\R}{{\mathbb R}}

\newcommand{\Z}{{\mathbb Z}}

\title[Decomposition of Jacobians and Shimura varieties]{Decomposition of Jacobian varieties and Shimura subvarieties in $A_g$}
\author{Abolfazl Mohajer}
\address{School of Mathematical and Statistical Sciences, University of Galway, Galway, Ireland}
\email{abmohajer83@ugmail.com}

\begin{document}
\begin{abstract}
Using the decomposition of Jacobians with group action, we prove the non-existence of some Shimura subvarieties in the moduli space of ppav $A_{g}$ arising from families of dihedral and quaternionic covers of the complex projective line $\P^1$. 
\end{abstract}
	
\maketitle
	
\section{Introduction}

Consider a family of non-singular complex projective curves $f:C\to T$ of genus $g$. The Torelli map associated to this family is the map $\tau:T\to A_g$ which sends a curve $[C]$ to its Jacobian $[J(C)]$. It is known by the classical Torelli theorem that this map is injective. The moduli space $A_g$ of $g$-dimensional complex principally polarized abelian varieties has the structure of a \emph{Shimura variety}. Indeed, $A_g=A_{g,l}=\Gamma_g(l)\backslash \H_g$, where $\H_g$ is the Siegel uper half-space of genus $g$ and $\Gamma_g(l)$ is the principal congruence subgroup of level-$l$ in $Sp_{2g}(\Z)$ which is the kernel of the natural map $Sp_{2g}(\Z)\to \Sp_{2g}(\Z/l)$, where $l\geq 3$ is an odd integer. In this case $\Gamma_g(l)$ is torsion-free and the quotient $\Gamma_g(l)\backslash \H_g$ is a smooth complex submanifold. Note that $\H_g=\Sp_{2g}(\R)/U(g)$. Inside $A_g$ there are specific subvarieties called \emph{special subvarieties} or \emph{Shimura subvarieties}. These are algebraic subvarieties of the form $\Gamma_g(l)\backslash\L$, where $\L$ comes from an algebraic subgroup $G(\R)\hookrightarrow \Sp_{2g}(\R)$. Special subvarieties have many interesting properties. They are for instance \emph{totally geodesic} subvarieties. In this paper, we consider families  $f:C\to T$ of $G$-Galois covers of $\P^1$ that generate special subvarieties in $A_g$. The special subvarieties arising from families of curves and in particular by Galois covers of $\P^1$  have been studied by many authors for instance in \cite{CFG}, \cite{CFGP}, \cite{DN}, \cite{M}, \cite{MO} and \cite{MZ}. Below, we will describe our main tools for studying such families in this paper.

\subsection{Group algebra decomposition of abelian varieties} For a more detailed presentation of the content of this subsection, we refer to the book of Birkenhake and Lange, \cite{BL}, chapter \S13.6. Another reference is \cite{LR2} by Lange and  Rodriguez. Let $A$ be an abelian variety over the complex numbers and suppose that a finite group $G$ acts on $A$. This induces an algebra homomorphism
\begin{equation} 
\rho:\Q[G]\to \en_{\Q}(A),
\end{equation} 
where $\Q[G]$ denotes the rational group algebra of $G$. This homomorphism induces a decomposition of the abelian variety $A$ upto isogeny. Although the representation $\rho$ may not be faithful, we do not distinguish between elements of $\Q[G]$ and their images in $\en_{\Q}(A)$. An element $\alpha\in\Q[G]$ defines an abelian subvariety $A^{\alpha}:=\ima(m\alpha)\subseteq A$, where $m$ is some positive integer such that $m\alpha\in\en(A)$. We remark that $\Q[G]$ is a finite dimensional semisimple $\Q$-algebra and therefore admits a unique decomposition
\begin{equation}\label{algdecomp}
\Q[G]=Q_1\times\cdots\times Q_r
\end{equation}
with simple $\Q$-algebras $Q_1, \ldots, Q_r$. The following theorem gives the decomposition of abelian varieties with a $G$-group action and is key to our further analysis. 
 \begin{theorem} \label{groupalgdecomp}
Let $G$ be a fintie group acting on an abelian variety $A$. Let $W_1, \ldots, W_r$ denote the irreducible $\Q$-representations of $G$ and $n_i:=\dim_{D_i}(W_i)$ with $D_i:=\en_G(W_i)$  for $i=1, \ldots, r$. Then there are abelian subvarieties $Y_1, \ldots, Y_r$ of $A$ and an isogeny. 
$A\sim  Y^{n_1}\times\cdots\times Y^{n_r}$.
\end{theorem} 
This decomposition is called the \emph{group algebra decomposition} of the abelian variety $A$ with respect to $G$. For more details about the above results we refer to  Birkenhake, Lange's book \cite{BL}, \S 13.6. In this article we are mainly interested in the case that $X$ is a smooth projective curve over $\C$ and $A=J(X)$ is the Jacobian of $X$ which is a principally polarized abelian variety over the complex numbers.\\


\subsection{Families of Galois covers of $\P^1$} In this paper we deal with families of Galois covers of the projective line $\P^1$. Let us briefly describe such covers. Consider a finite subset $\Delta=\{t_1,\dots, t_r\}\subset \mathbb{P}^1$. It is then well-known that the fundamental group $\pi_1(\P^1\setminus\Delta)$ has the following presentation
\begin{equation} \label{braid}
\pi_1(\P^1\setminus\Delta):= \Gamma_r=\langle g_1,\dots, g_r\mid g_1\cdots g_r=1 \rangle
\end{equation}

 If $C\to \P^1$ is a Galois cover of curves with branch locus $\Delta$, let us set $U:=\P^1\setminus\Delta$ and $V:=f^{-1}(U)$. Then $f|_V:V\to U$ is an unramified Galois covering. If $G$ is the Galois group of $f$, then there is an epimorphism $\pi_1(U)\to G$ or in other words an epimorphism $\Gamma_{r}\to G$. We therefore make following defintion.

\begin{definition}\label{datum}
A datum is a triple $(\mi, G, \Phi)$, where $\mi=(m_1,\ldots, m_r)$ is an $r$-tuple of integers $m_i\geq 2$, $G$ is a finite group and $\Phi:\Gamma_{r}\to G$ is an epimorphism such that $\Phi(\gamma_i)$ has order $m_i$ for each $i$.
\end{definition}
The Riemann’s existence theorem asserts that a cover $C\to \mathbb{P}^1$ can be completely determined by a datum. The following, is one of our main tools for excluding Shimura varieties generated by families of Galois covers of $\P^1$. 

\begin{lemma}
Let $\sC\to T$ be a family of $G$-covers of $\P^1$ given by the datum $(\mi, G, \Phi)$.  Suppose that $N\lhd G$ a normal subgroup and write $G^{\prime}=G/N$. Consider the family of quotients $\overline{C}=C/N\to\P^1$ of $G^{\prime}$-covers of $\P^1$ gcorresponding to the datum $(\mi^{\prime}, G^{\prime}, \Phi^{\prime})$. If $(\mi, G, \Phi)$ gives rise to a Shimura subvariety in the Torelli locus $T_g$, then $(\mi^{\prime}, G^{\prime}, \Phi^{\prime})$ is also a Shimura subvariety in $T_{g^{\prime}}$. 
\end{lemma}
\begin{proof}
This is a generalization of \cite{M}, Lemma 4.4. Let $f:C\to \mathbb{P}^1$ be a Galois $G$-cover of degree $|G|=n$. Consider the following diagram in which $\overline{f}:\overline{C}\to \mathbb{P}^1$ is the quotient cover with the notations as above. 
\[\begin{tikzcd}
C \arrow{rr}{f} \arrow[swap]{dr}{\varphi} & & \mathbb{P}^1 \\
& \overline{C} \arrow[swap]{ur}{\overline{f}}
\end{tikzcd}\]

If $\Delta$ (resp. $\overline{\Delta}$) is the branch locus of $f$ (resp. $\overline{f}$), then we have $\overline{\Delta}\subseteq\Delta$. Therefore if $T\subseteq\A^l$ and $T^{\prime}\subseteq\A^{l^{\prime}}$ denote the parametrizing locus of the families $\sC$ and $\overline{\sC}$ respectively and if we let $\rho:T\to T^{\prime}$ denote the projection which comes from the natural projection $\A^l\to \A^{l^{\prime}}$. Here $\overline{\sC}\to T^{\prime}$ is the family of curves associated with the datum $(\mi^{\prime}, G^{\prime}, \Phi^{\prime})$. Let $J\to T$ and $\overline{J}\to T^{\prime}$ be the corresponding relative Jacobians. Then $\rho^*\overline{J}$ is an isogeny factor of $J$. Let $\theta:T\to A_{g^{\prime}}$ be the morphism corresponding to $\rho^*\overline{J}$. Now this morphism is the composition of the projection $\rho:T\to T^{\prime}$ which is a dominant morphism and the morphism $T^{\prime}\to A_{g^{\prime}}$. Hence the closure of the image of $\theta$ is $Z^{\prime}$. 
\end{proof}

\begin{remark} \label{monodromydecomp}
 If the family $f:Y\rightarrow T$ gives rise to a Shimura subvariety in $A_{g}$ then the connected monodromy group $\mon^{0}$ is a normal subgroup of the generic Mumford-Tate group $M$. In fact in this case $\mon^{0}=M^{der}$. See for example \cite{R}. Consequently, if $f$ gives rise to a Shimura subvariety and $M^{ad}_{\mathbb{R}}=\prod_{1}^{l} Q_{i}$ as a product of simple  Lie groups then $\mon^{0,ad}_{\mathbb{R}}=\prod_{1}^{l} Q_{i}$.
 \end{remark}

\begin{remark} \label{smallestshimura}
Let $f:C\to T$ be a family of curves. Let $M$ be the generic Mumford-Tate group of the family $f:C\rightarrow T$. For the definition and construction of the generic Mumford-Tate group we refer to Rohde's book \cite{R}. Note that $M$ is a reductive $\mathbb{Q}$-algebraic group . Let $S_{f}$ be the natural Shimura variety associated to $M$. The Shimura subvariety $S_{f}$ is in fact the \emph{smallest} special subvariety that contains the image $Z$ of $T$ in $A_g$. The dimension $S_f$ depends on the real adjoint group $M^{ad}_{\mathbb{R}}$. Namely, if $M^{ad}_{\mathbb{R}}=Q_{1}\times\cdots\times Q_{r}$ is the decomposition of $M^{ad}_{\mathbb{R}}$ to $\mathbb{R}$-simple groups then $\dim S_{f}= \sum \delta(Q_{i})$. If $Q_{i}(\mathbb{R})$ is not compact then $\delta(Q_{i})$ is the dimension of the corresponding symmetric space associated to the real group $Q_{i}$ which can be read from Table V of \cite{H}. If $Q_{i}(\mathbb{R})$ is compact, i.e., if $Q_{i}$ is anisotropic we set $\delta(Q_{i})=0$. We remark that for $Q=PSU(p,q)$, $\delta(Q)=pq$ and for $Q=Psp_{2p}$, $\delta(Q)=\frac {p(p+1)}{2}$. Note also that $Z$ is a Shimura subvariety if and only if $\sum\delta(Q_{i})=l-3$, i.e., if and only if $\dim Z=\dim S_{f}=l-3$, where $l$ is the number of branch points of the covers in the family. According to Remark~\ref{monodromydecomp} above, if $M^{ad}_{\mathbb{R}}=\prod_{1}^{l} Q_{i}$ is the decomposition
 into $\mathbb{R}$-simple factors, then $\mon^{0,ad}_{\mathbb{R}}=\prod_{1}^{l} Q_{i}$ as well. We need to find eigenspaces
 $\mathcal{L}_{j_i}$ of types $\{a_i,b_i\}$ with $a_i$ and $b_i$ large enough in the sense described below.  Concretely, if we can find eigenspaces $\mathcal{L}_{j_i}$ as above, then these eigenspaces, being of types $\{a_i,b_i\}$, give rise to non-isomorphic $Q_i=\mon^{0,ad}_{\mathbb{R}}(\mathcal{L}_{j_i})=PSU(a_i,b_i)$ for $i\in K$ in the above decomposition of $\mon^{0,ad}_{\mathbb{R}}$ and if for $\delta(\mathcal{L}_{j_i})=\delta(\mon^{0,ad}_{\mathbb{R}}(\mathcal{L}_{j_i}))$, we have that $\sum \delta(\mathcal{L}_{j_i})>l-3$ (this is what we mean by \emph{large enough} in the above, note that $\delta(\mathcal{L}_{j_i})$ depends in our examples only on $a_i$ and $b_i$; see construction ~\ref{smallestshimura}), then
 $\dim S_f\geq \sum \delta(\mathcal{L}_{j_i})>l-3$.  
\end{remark}

\section{Shimura subvarieties}

\subsection{Families of dihedral covers} Throughout this note, we fix the following presentation for the dihedral group of order $2n$
\begin{equation}\label{dihed pres}
D_{n}=\langle r,s| r^{n}=s^2=(rs)^2=1\rangle
\end{equation}
The elements of the form $r^u$ with $1\leq u\leq n$ are called \emph{rotation} and elements of the form $r^as$ with $a\geq 0$ are called \emph{reflection}. Reflections are elements of order 2. We begin our analysis of dihedral covers by the following lemma.

\begin{lemma}\label{even}
Let $f:C\to \mathbb{P}^1$ be a $D_n$-cover. Then the number of branch points whose ramification is given by reflections $r^as$ as mentioned above is a non-zero even number. 
\end{lemma}
\begin{proof}
For this and more general results on dihedral and non-abelian covers we refer to \cite{M23}, especially Proposition 3.12. Note that we always assume that the cover is not ramified over the infinity. 
\end{proof}

 Let $C$ be a curve with $D_{2p}$-action, where $p$ is an odd prime. If $C\to Y=C/D_{2p}$ is the quotient cover, then by results of \cite{LR1}, \S 4.4 we have
\begin{equation}\label{dihed decomp}
J(C)\sim  J(Y)\times P(C_{\langle r\rangle}/Y)\times P(C_{D_p}/Y)\times P(C_{\widetilde{D}_p}/Y)\times B
\end{equation}
Here $D_{p}=\langle  r^{2}, s\rangle, \widetilde{D}_p=\langle  r^{2}, rs\rangle$ and $ \langle r\rangle$ are normal subgroups of index 2. The factor $B$ is a product of abelian varieties which we do not need in the sequel.


In the case that $n=2p$, with $p$ an odd prime number, one obtains the following statement about 1-dimensional families of dihedral covers.

In the following, we exclude all of the families of dimension greater than 12 and with a special ramifications type.

\begin{proposition} \label{dihed geq 3}
Suppose that we have a family of $D_{2p}$-covers of $\P^1$ with $l>12$ branch points (hence a family of dimension $\dim >9$) and such that all of the ramifications in $\langle r\rangle$ (i.e., given by rotations) are of the form $r^u$ with $u$ an odd number (the notation being as in Presentation \ref{dihed pres}). Then the image of this family under the Torelli map is not a Shimura subvariety. 
\end{proposition}
\begin{proof}
Suppose that the family is given by the branch points $(t_1,\ldots, t_l)$ with $l>9$. Now consider the the normal subgroup $N:=\langle r\rangle$ of index 2 and the quotient family of double covers $C_t/N\to \P^1$. Let $k$ be the number of branch points whose ramification is a reflection, i.e., of the form $r^as$ with $a\geq 0$.  Then the above quotient family corresponds to the family of covers of $\P^1$ with monodromy data $(2,(\underbrace{1,1,\ldots, 1}_{k}, 0,\ldots, 0))$, where the number of 1's is equal to $k$. This family has an eigenspace $\sL$ of the type $(\frac{k}{2}-1, \frac{k}{2}-1)$ and hence $\delta(\sL)=\frac{1}{2}(\frac{k^2}{4}-1)$, where $\delta(\sL)$ is introduced in Remark \ref{smallestshimura}. On the other hand considering the normal subgroup $D_p=\langle r^2, s\rangle$ of index 2, the quotient family $C_t/D_p\to \P^1$ yields a cover with the monodromy $(2,(\underbrace{1,1,\ldots, 1}_{\geq l-k}, 0\ldots, 0))$. As before, this family has an eigenspace $\sM$ with $\delta(\sM)\geq \frac{1}{2}(\frac{(l-k)^2}{4}-1)$ as in Remark \ref{smallestshimura}. Using the decomposition \ref{dihed decomp} and the above formula, we see that our in the decomposition of the adjoint group of our original family factors $\sL$ and $\sM$ occure and therefore the dimension of the smallest Shimura subvariety containing this family is at least $\delta(\sL)+\delta(\sM)$. Therefore if we have 
\begin{equation}\label{dimineq}
\delta(\sL)+\delta(\sM)\geq \frac{1}{2}(\frac{k^2}{4}-1)+\frac{1}{2}(\frac{(l-k)^2}{4}-1)>l-3,
\end{equation}
the family is not special. Now ~\ref{dimineq}, is equivalent to $k^2+(l-k)^2>8(l-2)$. This in turn yields the inequality $l^2-2(k+4)l+2(k^2+8)>0$. It is easy to verify that by our assumptions and conditions the above inequality holds true. 
\end{proof}

\begin{remark}
From the (1-dimensional) families of $D_{2p}$- covers of $\P^1$ in \cite{FGP} which give rise to Shimura subvariesties, none have the property that the ramifications  in $\langle r\rangle$ are all of the form  $r^u$ with $u$ odd. Therefore in view of Proposition \ref{dihed geq 3} one expects that such Shimura families should not exist. 
\end{remark}

\subsection{Families of $Q_8$-covers} 
Let $Q_8$ be the quaternion group, i.e., the following group
\[Q_8:=\langle i,j| i^2=j^2=(ij)^2=-1\rangle\]
If $X\to Y$ is a $Q_8$-cover, then 
\[J(X)\sim J(Y)\times P(X_{\langle i\rangle}/Y)\times P(X_{\langle j\rangle}/Y)\times P(X_{\langle ij\rangle}/Y)\times P(X/X_{\langle -1\rangle})\]
Consider the $Q_8$-cover given by the monodromy data $(2,4,4,4)$. 

\begin{equation} \label{commutative repres}
\begin{tikzcd}
C \arrow{rr}{} \arrow[swap]{dr}{4:1}& &\P^1\\
& C/\langle i\rangle  \arrow[swap]{ur}{2:1}& 
\end{tikzcd}  
\end{equation}
The intermediate cover $C\to C/\langle i\rangle$ is cyclic of order $4$. One checks that $g(C/\langle i\rangle)=0$ and so $C/\langle i\rangle\simeq \P^1$ and indeed the quotient cover $C/\langle i\rangle\to\P^1$ is a double cover ramified at two points. The same is true about $C/\langle j\rangle$ and $C/\langle ij\rangle$. 
Therefore we have that  $ P(C_{\langle i\rangle}/\P^1)= P(C_{\langle j\rangle}/\P^1)= P(C_{\langle ij\rangle}/\P^1))=J(\P^1)=0$. 
On the other hand 

\begin{equation} \label{commutative repres}
\begin{tikzcd}
C \arrow{rr}{} \arrow[swap]{dr}{2:1}& &\P^1\\
& C/\langle -1\rangle  \arrow[swap]{ur}{4:1}& 
\end{tikzcd}  
\end{equation}
We remark that $\langle -1\rangle$ is the center of $Q_8$ and that $Q_8/\langle -1\rangle\simeq \Z_2\times \Z_2$. So the $4:1$ cover $C/\langle -1\rangle\to \P^1$ is an abelian $\Z_2\times \Z_2$-cover. One checks that this cover has 3 branch points all of which of ramification index 2 and therefore $C/\langle-1\rangle\simeq \P^1$. The double cover $C\to C/\langle-1\rangle= \P^1$ is ramified over 10 points and therefore exhibits $C$ as a hyperelliptic curve of genus 4. Therefore $P(C/C_{\langle -1\rangle})=J(C)$ and the family is a Shimura curve contained in the hyperelliptic locus. Note that in this family, the Jacobians $J(C_t)$ are simple abelian varieties. Indeed one can show that this is the only family of simple Jacobians of $\Q_8$-covers up to isomorphism over a Shimura curve.
\begin{proposition}
A 1-dimensional family of $Q_8$-covers of $\P^1$ such that the Jacobian $J(C_t)$ is a  simple abelian variety is a hyperelliptic family and monodromy datum is $(2,4,4,4)$, i.e., it is isomorphic to the above family. 
\end{proposition}
\begin{proof}
Assume that we have a family such that the Jacobian $J(C_t)$ is simple and consider again the following diagram
\begin{equation} \label{commutative repres}
\begin{tikzcd}
C \arrow{rr}{} \arrow[swap]{dr}{2:1}& &\P^1\\
& C/\langle -1\rangle  \arrow[swap]{ur}{4:1}& 
\end{tikzcd}  
\end{equation}
Now since $J(C_t)$ is simple it holds that $C/\langle -1\rangle\simeq\P^1$, for otherwise $J(C/\langle -1\rangle)$ will be a non-trivial isogeny factor of $J(C_t)$. Consequently, the ramification over at least one branch point must be given by the element $(-1)$. One easily checks that the only possible monodromy indices are $(2,4,4,4)$. 
\end{proof}

\begin{proposition}
There does not exist a 1-dimensional Shimura family of non-hyperelliptic $Q_8$-covers of $\P^1$. 
\end{proposition}
\begin{proof}
Suppose that we have a 1-dimensional family of $Q_8$-covers of $\P^1$ whose fibers are non-hyperelliptic. We claim that are no branch points with ramification $(-1)$ for if there exists such a point, then $C/\langle -1\rangle\simeq\P^1$ and the cover $C\to C/\langle -1\rangle=\P^1$ exhibits $C$ as a hyperelliptic curve, a contradiction. This shows that the monodromy data of the family is of the form $(i,i,j,j)$ which corresponds to monodromy indices $(4,4,4,4)$. If $C\to\P^1$ is a $Q_8$-cover with the above monodromy, then by the Riemann-Hurwitz formula, $g(C)=5$. Then we have
\begin{equation} \label{commutative repres}
\begin{tikzcd}
C \arrow{rr}{} \arrow[swap]{dr}{4:1}& &\P^1\\
& C/\langle i\rangle  \arrow[swap]{ur}{2:1}& 
\end{tikzcd}  
\end{equation}
the quotient cover $C/\langle i\rangle\to\P^1$ is a double cover ramified at two points and hence $g(C/\langle i\rangle)=0$ and $C/\langle i\rangle\simeq\P^1$. The same is true also about the quotient cover $C/\langle j\rangle\to\P^1$. The quotient cover $C/\langle ij\rangle\to\P^1$ is the Legendre family of elliptic curves which is a special family. In particular, $g(C/\langle ij\rangle)=1$. Therefore by the above we have 
\[J(C)\sim C_{\langle ij\rangle}\times P(C/C_{\langle -1\rangle})\]
Therefore this family is not a Shimura family. Note that $C_{\langle ij\rangle}$ is a 1-dimensional Shimrua subvariety. Also, in the families found in \cite{GM} there are no Shimura families of Prym varieties of double covers of elliptic curves. This shows that the smallest Shimura subvariety that contains this family has dimension at least 2. 
\end{proof}

\end{document}